\documentclass[11pt]{amsart}

\newtheorem{theorem}{Theorem}[section]
\newtheorem{corollary}{Corollary}[section]
\newtheorem{lemma}{Lemma}[section]

\newtheorem{remark}{Remark}[section]

\newtheorem{definition}{Definition}[section]
\newtheorem{proposition}{Proposition}[section]
\usepackage{color}
\usepackage[dvipsnames]{xcolor}
\usepackage{soul}

\def \a{\alpha }

\def \l{\lambda }

\newcommand\restr[2]{{
  \left.\kern-\nulldelimiterspace
  #1
  \vphantom{\big|} 
  \right|_{#2} 
  }}

\begin{document}

\newcommand\nat{\mathbb N}
\newcommand\ganz{\mathbb Z}

\newcommand{\wta}{{\rm {wt} }  a }
\newcommand{\R}{\Bbb  R}

\newcommand{\wtb}{{\rm {wt} }  b }
\newcommand{\bea}{\begin{eqnarray}}
\newcommand{\eea}{\end{eqnarray}}
\newcommand{\be}{\begin {equation}}
\newcommand{\ee}{\end{equation}}
\newcommand{\g}{\frak g}
\newcommand{\hg}{\widehat {\frak g} }
\newcommand{\hn}{\hat {\frak n} }
\newcommand{\h}{\frak h}
\newcommand{\V}{\Cal V}
\newcommand{\hh}{\hat {\frak h} }
\newcommand{\n}{\frak n}
\newcommand{\Z}{\Bbb Z}
\newcommand{\N}{{\Bbb Z} _{> 0} }
\newcommand{\Zp} {\Z _ {\ge 0} }
\newcommand{\Hp}{\bar H}
\newcommand{\C}{\Bbb C}
\newcommand{\Q}{\Bbb Q}
\newcommand{\1}{\bf 1}
\newcommand{\la}{\langle}
\newcommand{\ra}{\rangle}
\newcommand{\NS}{\bf{ns} }

\newcommand{\wt}{{\rm {wt} }   }

\newcommand{\E}{\mathcal E}
\newcommand{\F}{\mathcal F}
\newcommand{\X}{\bar X}
\newcommand{\Y}{\bar Y}

\newcommand{\hf}{\mbox{$\tfrac{1}{2}$}}
\newcommand{\thf}{\mbox{$\tfrac{3}{2}$}}

\newcommand{\W}{\mathcal{W}}
\newcommand{\non}{\nonumber}
\def \l {\lambda}
\baselineskip=14pt
\newenvironment{demo}[1]%
{\vskip-\lastskip\medskip
  \noindent
  {\em #1.}\enspace
  }%
{\qed\par\medskip
  }

\def \l {\lambda}
\def \a {\alpha}

\title[On fusion rules for the Weyl vertex algebra]{On fusion rules and intertwining operators  for the Weyl vertex algebra}
 
\author[]{Dra\v zen  Adamovi\' c}
\author[]{Veronika Pedi\' c}
 
\date{\today}
\begin{abstract}
In vertex algebra theory, fusion rules are described as the dimension of the vector space of intertwining operators between three irreducible modules. We describe fusion rules in the category of weight modules for the Weyl vertex algebra. This way we confirm the conjecture on fusion rules based on the Verlinde algebra. We explicitly construct intertwining operators appearing in the formula for fusion rules.  We present a result which relates irreducible weight  modules for the Weyl vertex algebra to the irreducible modules for the affine Lie superalgebra $\widehat{gl(1 \vert 1)}$.
\end{abstract}
\maketitle

\section{Introduction}

In the theory of vertex algebras and conformal field  theory, determination of fusion rules is one of the most important problems. By a result by Y. Z. Huang \cite{H} for  a rational vertex algebra, fusion rules can be determined by using the Verlinde formula. However, although there are certain versions of Verlinde formula for a broad class of non-rational vertex algebras, so far there is no proof that fusion rules for such algebras can be determined by using the Verlinde formula. One important example is the singlet vertex algebra for $(1,p)$--models whose irreducible representations were classified in \cite{A-2003}. Verlinde formula for fusion rules was also presented by T. Creutzig and A. Milas in  \cite{CM}, but so far the proof was only given for the case $p=2$ in  \cite{AdM-2017}.
We should also mention that the fusion rules  and intertwining operators for some affine and superconformal  vertex algebras were  studied in  \cite{A-2001}, \cite{A-CMP},  \cite{CRTW} and \cite{KR-CMP}.

In this paper we study the case of the Weyl vertex algebra, which we denote by $M$, also called the $\beta \gamma$ system in the physics litarature. Its Verlinde type conjecture for fusion rules was presented by S. Wood and D. Ridout in \cite{RW}. Here, we present a short proof of Verlinde conjecture in this case. We prove the following fusion rules result:
 \begin{theorem} Assume that $\lambda , \mu, \lambda + \mu \in  {\C} \setminus {\Z}$. Then we have:
 \bea && \rho_{\ell_1} (M) \times    \rho_{\ell_2 }  (\widetilde{U(\lambda)} )   =  \rho_{\ell_1+\ell_2 }  (\widetilde{U(\lambda})),  \label{fusion-uv-1}  \\
&&  \rho_{\ell_1}  (\widetilde{U(\lambda )} )  \times   \rho_{\ell_2 }  (\widetilde{U(\mu )} )  =   \rho_{\ell_1+\ell_2 }  (\widetilde{U(\lambda  + \mu)} )   + \rho_{\ell_1+\ell_2 -1}  (\widetilde{U(\lambda  + \mu)} ),    \label{fus-uv-2}  \eea
where  $\widetilde{U(\lambda )} $ is an irreducible weight module, and $\rho_{\ell}$, $\ell \in {\Z}$,  are the spectral flow automorphisms defined by  (\ref{sepctral-flow-def}).
\end{theorem}
The fusion rules  (\ref{fusion-uv-1}) was proved in Proposition \ref{sc-fusion-1}, and it is a direct  consequence of the construction of H. Li \cite{Li-1997}. 
 The  main contribution of our paper is vertex-algebraic  proof  of (\ref{fus-uv-2}) which  uses   the theory of intertwining operators for vertex algebras and the fusion rules for the affine vertex superalgebra $V_1(gl(1 \vert 1))$.

  We also prove a general irreducibility result which relates irreducible weight modules for the Weyl vertex algebra $M$ to irreducible weight modules for $V_1(gl(1 \vert 1))$ (see Theorem  \ref{ired-general}).
 
 \begin{theorem}
 Assume that $\mathcal N$ is an irreducible weight $M$--module. Then $\mathcal N \otimes F$ is a completely reducible  $V_1(gl(1 \vert 1))$--module.
 \end{theorem}

 The construction of intertwining operators appearing in the fusion rules is based on two different embeddings of the Weyl vertex algebra $M$ into  the lattice vertex algebra $\Pi(0)$. Then one  $\Pi(0)$--intertwining operator gives two different $M$--intertwining operators. Therefore, both intertwining operators are realized as $\Pi(0)$--intertwining operators. Once we tensor the Weyl vertex algebra $M$ with the Clifford vertex algebra $F$, we can use the fusion rules for $V_1(gl(1 \vert 1))$ to calculate the fusion rules for $M$.

  It is known that fusion rules can be determined by using fusion rules for the singlet vertex algebra (cf. \cite{AdM-2017}, \cite{CM}). However, we believe that our methods, which use $V_1(gl(1 \vert 1) )$, can be generalized to a wider class of vertex algebras. In our future work we plan to study the following related  fusion rules problems:
  
  \begin{itemize}
  \item Connect fusion rules for higher rank Weyl vertex algebra with fusion rules for $V_1(gl(n \vert m))$.
  \item Extend fusion ring with weight modules having infinite-dimensional weight spaces (cf. Subsection \ref{more-modules}) and possibly with irreducible  Whittaker modules (cf.  \cite{ALPY-2018}).

  \end{itemize}

\vskip 5mm
{\bf Acknowledgment.}

This work was done in
part during the authors’ stay at the Research Institute for Mathematical Sciences
(RIMS), Kyoto in July 2018. 
 The results of this paper were also  reported by V. P. at the Southeastern Lie Theory Workshop X
University of Georgia,  Athens, GA, June 10--12, 2018.

The authors  are  partially supported   by the
QuantiXLie Centre of Excellence, a project coffinanced
by the Croatian Government and European Union
through the European Regional Development Fund - the
Competitiveness and Cohesion Operational Programme
(KK.01.1.1.01.0004).

\section{Fusion rules and intertwining operators}

In this section we recall the definition of intertwining operators and fusion rules. More details can be found in  \cite{FZ}, \cite{FHL}, \cite{DLM}, \cite{CHY}. We also prove an important result on the action of certain automorphisms on intertwining operators. This result will enable us to produce new intertwining operators from the existing one.

Let $V$ be a conformal vertex algebra with the conformal vector $\omega$ and let $Y(\omega, z)= \sum_{n \in {\Z}} L(n) z^{-n-2}$. We assume that the derivation in the vertex algebra $V$ is $D=L(-1)$. A $V$-module (cf. \cite{LL}) is a vector  
space $M$  endowed with
 a linear map $Y_M$ from
$V$ to the space of $End(M)$-valued fields
$$ a\mapsto Y_M(a,z)=\sum_{n\in\Z}a^M_{(n)}z^{-n-1}
$$ such that:
\begin{enumerate}
\item $Y_M(|0\rangle,z)=I_M$,
\item for $a,b \in V$,
\begin{align*}& z_0^{-1}\delta\Big(\frac{z_1-z_2}{z_0}\Big)Y_M(a,z_1)Y_M(b,z_2) - z_0^{-1}\delta\Big(\frac{z_2-z_1}{-z_0}\Big)Y_M(b,z_2)Y_M(a,z_1)\\
    &= z_2^{-1}\delta\Big(\frac{z_1-z_0}{z_2}\Big)Y_M(Y(a,z_0)b,z_2).
\end{align*}
\end{enumerate}

Given three $V$-modules $M_1$, $M_2$, $M_3$, an {\it intertwining operator of type $\Big(\begin{matrix}M_3\\M_1\ M_2\end{matrix}\Big)$} (cf.   \cite{FHL}, \cite{FZ}) is a   map $I:a\mapsto I(a,z)=\sum_{n\in \ganz}a^I_{(n)}z^{-n-1}$ from
$M_1$ to the space of $Hom(M_2,M_3)$-valued fields such that:
\begin{enumerate}
\item for $a \in V$, $b \in M_1$, $c\in M_2$, the following Jacobi identity holds:
\begin{align*}&  z_0^{-1}\delta\Big(\frac{z_1-z_2}{z_0}\Big)Y_{M_3}(a,z_1)I(b,z_2)c - z_0^{-1}\delta\Big(\frac{z_2-z_1}{-z_0}\Big)I(b,z_2)Y_{M_2}(a,z_1)c\\
    &= z_2^{-1}\delta\Big(\frac{z_1-z_0}{z_2}\Big)I(Y_{M_{1}}(a,z_0)b,z_2)c,  
\end{align*}
\item for every $a \in M_1$, \[ I(L(-1)a,z) = \frac{d}{dz}I(a, z).\] 
\end{enumerate}

We let $I { M_3 \choose M_1 \ M_2}$ denote the space of intertwining operators of type ${ M_3 \choose M_1 \ M_2}$, and set 
$$N^{M_3}_{M_1,M_2}=\dim I\left(\begin{matrix}M_3\\M_1\ M_2\end{matrix}\right).$$ When $N^{M_3}_{M_1,M_2}$ is finite, it is  usually called a {\it fusion coefficient}.\par

 Assume that in the category 
$K$ of $L(0)$-diagonalizable
$V$-modules, irreducible modules  $\{M _i \ \vert \ i \in I\}$, where $I$ is an index set, have the following properties
\begin{itemize}
\item[(1)]
  for every $i, j \in I$,  $N^{M_k}_{M_i,M_j}$ is finite for any $k\in I$;
  \item[(2)]  $N^{M_k}_{M_i,M_j} = 0 $ for all but finitely many $k \in I$.
  \end{itemize}
  
Then the algebra with basis $\{e_i  \in I \}$ and product
$$e_i\cdot e_j = \sum_{k\in I } N^{M_k}_{M_i,M_j}  e_k$$
is called the {\it fusion algebra} of $V, K$.

\vskip5pt

Let $K$ be a category of  $V$-modules. Let $M_1$, $M_2$ be irreducible $V$-modules in $K$. Given an irreducible $V$-module $M_3$ in $K$, we will say that the fusion  rule
\begin{equation}\label{fr-sc}
  M_1 \times M_2=M_3
  \end{equation}
holds in $K$ if $N^{M_3}_{M_1,M_2}=1$ and  $N^{R}_{M_1,M_2}=0$
   for any other irreducible $V$-module  $R$ in $K$ which is not isomorphic to $M_3$.

We say that an irreducible $V$-module  $M_1$ is a simple current in $K$   if $M_1$ is in $K$ and, for every irreducible $V$-module  $M_2$ in $K$, there is an irreducible $V$-module $M_3$ in $K$, such that the fusion rule \eqref{fr-sc} holds in $K$
 (see \cite{DLM}).
 
Recall that for any automorphism $g$ of $V$, and any $V$--module $(M,Y_M(\cdot, z))$, we have a new $V$--module $M \circ g=M^g$, such that $M^g \cong M$ as a vector space and the vertex operator $Y_M^g$ is given by $Y_M^g(v, z) := Y_M(g v,z)$, for $v \in V$. Namely, the only axiom we have to check is the Jacobi identity, and we have:
\begin{align*}
&z_0^{-1}\delta\Big(\frac{z_1-z_2}{z_0}\Big)Y_M^g(a,z_1)Y_M^g(b,z_2) - z_0^{-1}\delta\Big(\frac{z_2-z_1}{-z_0}\Big)Y_M^g(b,z_2)Y_M^g(a,z_1)\\
&= z_0^{-1}\delta\Big(\frac{z_1-z_2}{z_0}\Big)Y_M(ga,z_1)Y_M(gb,z_2) - z_0^{-1}\delta\Big(\frac{z_2-z_1}{-z_0}\Big)Y_M(gb,z_2)Y_M(ga,z_1) \\
& = z_2^{-1}\delta\Big(\frac{z_1-z_0}{z_2}\Big)Y_M(Y(ga,z_0)gb,z_2)\\
& = z_2^{-1}\delta\Big(\frac{z_1-z_0}{z_2}\Big)Y_M^g(Y^g(a,z_0)b,z_2).
\end{align*}
Therefore, $M^g$ is a $V$--module.
The following proposition shows that automorphism $g$ also produces a new intertwining operator. 

\begin{proposition}\label{auto} Let $g$ be an automorphism of the vertex algebra $V$ satisfying the condition
\bea \label{uvjet-autom} &&  \omega - g(\omega) \in \mbox{Im} (D).\eea
Let 
 $M_1$, $M_2$, $M_3$ be $V$--modules and $I(\cdot, z)$ an intertwining operator of type  ${ M_3 \choose M_1 \ M_2}$. Then we have an intertwining operator $I^g$ of type ${ M_3 ^ g \choose M_1 ^ g  \ M_2 ^ g}$, such that
$I^g (b, z_1) = I(b, z_1)$, for all $b \in M_1$. Moreover,
\[ N^{M_3}_{M_1,M_2} = N^{M_3 ^ g }_{M_1 ^ g ,M_2 ^  g}. \]
\end{proposition}
\begin{proof} 
We have:
\begin{align*}
&z_0^{-1}\delta\Big(\frac{z_1-z_2}{z_0}\Big)Y_{M_3}^g(a,z_1)I^g(b,z_2)c - z_0^{-1}\delta\Big(\frac{z_2-z_1}{-z_0}\Big)I^g(b,z_2)Y_{M_2}^g(a,z_1)c\\
&= z_0^{-1}\delta\Big(\frac{z_1-z_2}{z_0}\Big)Y_{M_3}(ga,z_1)I(b,z_2)c - z_0^{-1}\delta\Big(\frac{z_2-z_1}{-z_0}\Big)I(b,z_2)Y_{M_2}(ga,z_1)c\\
&= z_2^{-1}\delta\Big(\frac{z_1-z_0}{z_2}\Big)I(Y_{M_{1}}(ga,z_0)b,z_2)c\\ 
&= z_2^{-1}\delta\Big(\frac{z_1-z_0}{z_2}\Big)I(Y_{M_{1}}^g(a,z_0)b,z_2)c.
\end{align*}

Set $$Y^g (\omega, z) = \sum_{n \in {\Z}} L(n) ^g z^{-n-1}. $$
Since $g(\omega) = \omega + D v$ for certain $v \in V$, we have that
$$ g(\omega) _0 = \omega_0 + (Dv)_0 = \omega_0 = L(-1). $$
This implies that
$ L(-1)^g = L(-1)$.
Hence for $a \in M_1$ we have
$$ I^g ( L(-1) ^g  a, z)  =  I^g ( L(-1)   a, z)   =I (L(-1) a, z) = \frac{d}{d z} I(a, z) =  \frac{d}{d z} I^g (a, z).  $$
Therefore, $I^g$ has the $L(-1)$--derivation property and $I^g$ is an intertwining operator of type ${ M_3 ^ g \choose M_1 ^ g  \ M_2 ^ g}$.

\end{proof}

\begin{remark}
 	If $V$ is a vertex operator algebra and $g$ an automorphism of $V$, then $g(\omega)= \omega$ and the condition (\ref{uvjet-autom}) is automatically satisfied. In our applications, $g$ will only be a vertex algebra automorphism such that $g(\omega) \ne \omega$, yet the condition (\ref{uvjet-autom}) will be satisfied.
\end{remark}

\section{The Weyl vertex algebra}
\label{Weyl}

\subsection{The Weyl vertex algebra}
The  {\it Weyl algebra} $\widehat{\mathcal A}$ is an associative algebra with generators
$$ a(n), a^{*} (n) \quad (n \in {\Z})$$ and relations
\bea  && \  \  [a(n), a^{*} (m)] = \delta_{n+m,0}, \ \  [a(n), a(m)] = [a ^{*} (m), a ^* (n) ] = 0 \ \ (n,m \in {\Z}).   \label{comut-Weyl}  \eea

Let $M$ denote the simple {\it Weyl module} generated by the cyclic vector ${\bf 1}$ such that
$$ a(n) {\bf 1} = a  ^* (n+1) {\bf 1} = 0 \quad (n \ge 0). $$
As a vector space, $$ M \cong {\C}[a(-n), a^*(-m) \ \vert \ n > 0, \ m \ge 0 ]. $$

 There is a unique vertex algebra $(M, Y, {\bf 1})$  (cf. \cite{FB}, \cite{KR},  \cite{efren}) where
the  vertex operator map is $$ Y: M \rightarrow \mbox{End}(M) [[z, z ^{-1}]] $$
such that
$$ Y (a(-1) {\bf 1}, z) = a(z), \quad Y(a^* (0) {\bf 1}, z) = a ^* (z),$$
$$ a(z)   = \sum_{n \in {\Z} } a(n) z^{-n-1}, \ \ a^{*}(z) =  \sum_{n \in {\Z} } a^{*}(n)
z^{-n}. $$

In particular we have:
$$ Y (a(-1) a^*  (0) {\bf 1}, z) =  a(z) ^+ a^* (z)  +  a ^* (z)  a(z) ^- , $$
where
$$a (z) ^+ = \sum_{n \le -1} a(n) z ^{-n-1}, \quad a (z) ^- = \sum_{n \ge  0} a(n) z ^{-n-1}. $$

 Let $\beta := a(-1) a^*  (0) {\bf 1}$. Set $\beta (z) = Y(\beta ,z) =  \sum_{n \in {\Z} } \beta (n) z ^{-n-2}$. Then $\beta$ is a Heisenberg vector in $M$ of level $-1$. This means that the components of the field $\beta(z)$ satisfy the commutation relations
$$[ \beta(n), \beta(m)] = - n \delta_{n+m,0} \quad (n, m \in {\Z}). $$
 Also, we have the following formula
$$ [\beta(n), a(m) ] = - a (n+m), \quad  [\beta(n), a^* (m) ] = a^* (n+m). $$

The vertex algebra $M$ admits a family  of Virasoro vectors
$$ \omega_{\mu} = (1-\mu) a  (-1) a^*  (-1) {\bf 1} - \mu a (-2) a^{*} (0) {\bf 1}  \quad (\mu \in {\C})$$
of central charge $c_{\mu} = 2 (6 \mu (\mu-1) +1)$. Let 
$$L^{\mu} (z) = Y(\omega, z) = \sum_{n \in {\Z} } L^{\mu} (n) z^{-n-2}.$$
This means that the components of the field $L(z)$ satisfy the relations
$$[L^{\mu} (n), L^{\mu} (m) ] = (n-m) L^{\mu} (m+n) + \frac{n^3 - n} {12} \delta_{n+m,0} c_{\mu}.$$

For $\mu =0$, we write $\omega = \omega_{\mu}$, $L^{\mu} (n) = L(n)$. Then $c_{\mu} = 2$. Clearly
$$\omega_{\mu} = \omega - \mu \beta(-2) {\bf 1}. $$

Since $(\beta(-2)\textbf{1})_0 = (D\beta)_0$, we have that 

\bea &&  L^{\mu}(-1) = L(-1), \quad \mbox{for every} \ \mu \in \mathbb{C}. \label{der-mu}   \eea

For $n, m \in {\Z}$ we have
$$ [L(n), a(m) ] = - m a(n+m), \quad [L(n), a^*(m) ] = - (m+ n) a^*(n+m). $$
In particular,
$$[L(0), a(m) ] =-m a(m), [L(0), a^*(m)]=-m a^*(m). $$


\begin{lemma} \label{classification-zhu}
	Assume that $W = \bigoplus_{\ell \in {\Z_{\ge 0} }  } W(\ell) $ is a ${\Z}_{\ge 0}$--graded $M$--module with respect to $L(0)$.  Then $$L(0) \equiv 0 \quad \mbox{on} \ W(0). $$
\end{lemma}
\begin{proof}
Since $W$ is ${\Z}_{\ge 0}$ with top component $W(0)$, the operators $a(n), a^*(n)$ must act trivially on $W(0)$ for all $n \in {\Z}_{>0}$.  Since
$$ L(z) = :\partial a^*(z)  a (z): = \sum_{n \in {\Z} } L(n) z ^{-n-2}, $$
we have
$$ L(0) = \sum_{n= 1} ^{\infty}  n (  a ^* (-n) a (n)   - a(-n) a^* (n)) \equiv 0 \quad (\mbox{on} \ W(0)). $$
The Lemma holds. 
	\end{proof}
	
	\begin{remark}
	One can show that Zhu's algebra $A(M) \cong A_1$ and that $[\omega] = 0$ in $A(M)$. This can give a second proof of Lemma \ref{classification-zhu}.
	\end{remark}

\begin{definition}
A module $W$ for the Weyl algebra $\widehat{\mathcal A}$ is called restricted if the following condition holds:
\begin{itemize}
\item For every $w \in W$, there is $N \in {\Z}_{\ge 0}$ such that
$$ a(n) w = a^* (n) w = 0, \quad \mbox{for} \ n \ge N. $$
\end{itemize}
\end{definition}

 \subsection{Automorphisms of the Weyl algebra} 
 
 Denote by $\mbox{Aut} (\widehat{ {\mathcal A}})$ the group of automorphisms of the Weyl algebra $\widehat{ {\mathcal A}}$. For any $f \in \mbox{Aut} (\widehat{ {\mathcal A}})$, and  $\widehat{ {\mathcal A}}$--module $N$, one can construct $\widehat{ {\mathcal A}}$--module $f(N)$ as follows:
 $$ f(N) := N \quad \mbox{as vector space, and action is} \ \ x. v = f(x) v \quad (v \in N).$$
 For $f, g \in  \mbox{Aut} (\widehat{ {\mathcal A}})$,  we have
\bea \label{composit-action}    (f \circ g) (N) = g (f(N)). \eea
 For every $ s \in {\Z}$ the Weyl  algebra $\widehat{ {\mathcal A}}$ admits the following automorphism
\bea   \rho_s ( a(n) ) = a (n+s), \quad  \rho_s ( a^* (n) ) = a^*  (n-s).  \label{sepctral-flow-def} \eea
 Then $\rho_s$ is an automorphism of  $\widehat{ {\mathcal A}}$ which can be lifted to an automorphism of the vertex algebra $M$. Automorphism $\rho_s$ is called spectral flow automorphism.
 
 Assume that $U$ is any restricted module for $\widehat{ {\mathcal A}}$. Then $\rho_s(U)$ is also a restricted  module for  $\widehat{ {\mathcal A}}$ and $\rho_s(U)$ is a module for the vertex algebra $M$.

Let $\mathcal K$, be the category of weight $M$--modules such that the operators $\beta(n)$, $n \ge 1$, act locally nilpotent on each module $N$  in $\mathcal K$. Applying the automorphism $\rho_s$ to the vertex algebra $M$, we get $M$--module $\rho_s(M)$, which is a simple current in the category $\mathcal K$. The proof is essentially given by H. Li in \cite[Theorem 2.15]{Li-1997} in a slightly different setting.  
 
\begin{proposition} \label{sc-fusion-1}  \cite{Li-1997}
Assume that $N$ is an irreducible weight $M$--module in the category $\mathcal K$. Then the following fusion rules hold:
\bea  \rho_{s_1} (M) \times \rho_{s_2} (N)  =  \rho_{s_1+ s_2} (N) \quad (s_1, s_2 \in {\Z}). \label{fusion-sc} \eea
\end{proposition}
\begin{proof}
First we notice that  if $N$ is an irreducible $M$--module in $\mathcal K$, we have the following fusion rules
\bea  M \times N = N. \label{fusion-basic} \eea
Using \cite{Li-1997}, one can prove that $\rho_s(M)$ is constructed from $M$ as:
$$ (\rho_s(M), Y_s (\cdot,z)):= (M, Y( \Delta(-s \beta, z)\cdot, z)),$$
where
$$ \Delta(v, z):=z^{v_0} \exp\left( \sum_{n=1}^{\infty} \frac{v_n}{-n} (-z) ^{-n}\right). $$
Assume that $N_i$, $i=1,2,3$ are irreducible modules in $\mathcal K$. By  \cite[Proposition 2.4]{Li-1997}   from an intertwining operator   $I(\cdot, z)$  of type  ${ N_3 \choose N_1 \ N_2}$,  one  can construct intertwining operator   $I_{s_2}(\cdot, z)$  of type
${  \rho_{s_2} (N_3) \choose N_1 \ \rho_{s_2} (N_2)}$, where
$$I_{s_2} (v, z):= I(\Delta(-s_2 \beta, z) v, z) \quad (v \in  N_1 ). $$
Now, the fusion rules (\ref{fusion-sc}) follows easily  from the above construction using (\ref{fusion-basic}). 
\end{proof}

Consider the following automorphism of the Weyl vertex algebra
\bea  g : && M \rightarrow M  \nonumber \\
                   &&  a\mapsto  -a^{*}, \ \ 
                    a^*   \mapsto   a \nonumber \eea 
 Assume that $U$ is any $M$--module.
 Then $U^{g }= U \circ g $ is generated by the following fields
 $$ a_{g} (z) = -  a^*(z), \ a^ * _{g} (z) =   a(z). $$
 As an $\widehat {\mathcal A}$--module, $U^{g}$ is obtained from $U$ by applying the following automorphism $g$:
  \bea &&  a(n) \mapsto - a^* (n+1), \ \ a^{*} (n) \mapsto a(n- 1). \label{autom-1} \eea
 This implies that 
\bea && g =  \rho_{-1} \circ \sigma = \sigma \circ \rho_{1} \label{autom-2} \eea where $\sigma$ is the  automorphism of  $\widehat {\mathcal A}$ determined by
\bea && a(n) \mapsto - a^* (n), \ \ a^{*} (n) \mapsto a(n). \label{autom-3} \eea
 The automorphism $g$ is then a vertex algebra automorphism of order $4$.
 Denote by $\sigma_0 $ the restriction of $\sigma$ on $A_1$.
Using (\ref{autom-1})-(\ref{autom-3})  we get  the following result:
 
 \begin{lemma} \label{djelovanje-autom}
  Assume that $W$ is an irreducible   $M$--module.
  Then $$W^g \cong \rho_{1} (\sigma (W) ). $$
 \end{lemma}
 \begin{proof}
We have that as an $\widehat{\mathcal A}$--module:
 $$  W^ g =     (\rho_{-1} \circ \sigma)  (W). $$ 
 Since $\rho_{-1} \circ \sigma = \sigma \circ \rho_1$, by applying  (\ref{composit-action}) we get:
$$ W^ g =   (\sigma \circ \rho_1 )(W) = \rho_1 (\sigma (W)).$$
The proof follows.
 \end{proof}

\subsection{Weight modules for the Weyl vertex algebra}

\begin{definition} A module $W$ for the Weyl vertex algebra $M$ is called {\bf weight} if the operators $\beta(0)$ and $L(0)$ act semisimply on $W$. 
\end{definition}

 Clearly, vertex algebra $M$ is a weight module. We will now construct a family of weight modules.
 
 \begin{itemize}
 \item Recall that the first Weyl algebra $A_1$ is generated by $x, \partial _x$ with the commutation relation
 $$[\partial_x, x ] = 1. $$
 \item For every $\lambda \in {\C}$, $$U(\lambda) :=x^{\lambda} {\C}[x, x^{-1}]$$ has the structure of an  $A_1$--module.
 
 \item $U(\lambda)$ is irreducible if and only if $\lambda \in {\C} \setminus {\Z}$.
 
 \item Note that $a(0), a^*(0)$ generate a subalgebra of the Weyl algebra, which is isomorphic to the first Weyl algebra $A_1$. Therefore $U(\lambda)$ can be treated as an $A_1$--module by letting $a(0) = \partial_x$, $a^*(0) = x$. 
 
  \item  By applying the automorphism $\sigma_0$ on $U(\lambda)$ we get
 $$ \sigma_0  (U(\lambda)) \cong U(-\lambda). $$
Indeed, let   $z_{-\mu-1}  = \frac{x^{\mu}} {\Gamma(\mu+1)}$, where  $\mu \in {\C}\setminus{\Z}$. Then
$$ \sigma_0( a(0)). z_{-\mu}  = - \frac{x^{\mu}} {\Gamma(\mu)} = -   \mu  \frac{x^{\mu}} {\Gamma(\mu +1)} = -\mu z_{-\mu -1}. $$
$$ \sigma_0( a^* (0)). z_{-\mu}  =  (\mu-1) \frac{x^{\mu-2}} {\Gamma(\mu)} = \frac{x^{\mu-2}} {\Gamma(\mu-1)} = z_{-\mu +1}  . $$

\item  Define the following subalgebras of $\widehat{\mathcal A}$:
 $$ \widehat{\mathcal A}_{\ge 0 } =   {\C}[a( n), a^*( m) \ \vert \ n, m \in {\Z}_{\ge 0}  ], $$
  $$ \widehat{\mathcal A}_{<  0 } =   {\C}[a( - n), a^*( - n) \ \vert \ n \in {\Z}_{\ge 1}  ]. $$

 \item The $A_1$--module structure on $U(\lambda)$ can be extended to a structure of $\widehat{\mathcal A}_{\ge 0 }$--module by defining
 $$ \restr{a(n)}{U(\lambda)} = \restr{a^*(n)}{U(\lambda)} \equiv 0 \qquad (n \ge 1). $$

 \item Then we have the  induced module for the Weyl algebra:
 $$\widetilde{U(\lambda)} =  \widehat{\mathcal A} \otimes  _{.    \widehat{\mathcal A}_{\ge 0 }     }  U(\lambda)$$
 \item[] which is isomorphic to
 $$ {\C}[a (-n), a^*(-n) \ \vert n \ge 1] \otimes U(\lambda)$$
as a vector space.

 \end{itemize}
 
 \begin{proposition}
 For every $\lambda \in {\Bbb C} \setminus {\Z}$, $\widetilde{U(\lambda)}$ is an irreducible weight module for the Weyl vertex algebra $M$.
  \end{proposition}
 \begin{proof}
 The proof follows from Lemma \ref{classification-zhu} and the fact that 	$\widetilde{U(\lambda)}$ is a ${\Z_{\ge 0}}$--graded $M$--module whose top component is an irreducible module for $A_1$.
 	\end{proof}
 Applying Lemma  \ref{djelovanje-autom} we get:
   \begin{corollary} \label{ired-relaxed}
 For every $\lambda \in {\Bbb C} \setminus {\Z}$  and $s \in {\Z}$ we have  $$\widetilde{U(\lambda)} ^{g} \cong \rho_{1} ( \widetilde{U(-\lambda)}), \quad \left( \rho_{-s+1} ( \widetilde{U(\lambda)}) \right)^g  \cong \rho_{s} ( \widetilde{U(-\lambda)}).  $$
  \end{corollary}

\subsection{More general weight modules}
\label{more-modules}
A  classification of irreducible  weight modules for the Weyl algebra  $\widehat{\mathcal A}$ is given in \cite{FGM}.  Let us describe here a family of weight modules having infinite--dimensional weight spaces.

Take ${\lambda}, \mu \in {\Bbb C} \setminus {\Z}$. Let $$U(\lambda, \mu) = x_1 ^{\lambda} x_2 ^{\mu} {\C} [x_1, x_2, x_1 ^{-1}, x_2 ^{-1}]. $$ Then $U(\lambda, \mu)$ is an irreducible module for the second Weyl algebra $A_2$ generated by $\partial _1, \partial_2, x_1, x_2$. Note that $A_2$ can be realized as a subalgebra of  $\widehat{\mathcal A}$ generated by $\partial_2 = a(1), \partial _1 = a(0),  x_2 = a^*(-1), x_1 = a^*(0)$. Then we have the irreducible $\widehat{\mathcal A}$--module $\widetilde{U( \lambda, \mu)} $ as follows.
Let  $\mathcal B$ be the subalgebra of $\widehat{\mathcal A}$ generated by $a(i), a^*(j)$, $i \ge 0, j \ge -1$. Consider $U(\lambda, \mu)$ as a $\mathcal B$--module such that
$ a(n) $, $a^*(m)$ act trivially for $n \ge 2$, $m \ge 1$. Then by \cite{FGM},
$$\widetilde{U( \lambda, \mu)}  =\widehat{\mathcal A}  \otimes _{\mathcal B} U(\lambda, \mu)$$
is an irreducible $\widehat{\mathcal A}$--module. 
As a vector space:
\bea \widetilde{U( \lambda, \mu)} &\cong &  {\C}[ a (-n-1), a^{*} (-m-2) \ \vert \ n,m  \in {\Z}_{\ge   0}]   \otimes U(\lambda, \mu) \nonumber \\   &\cong&    a^* (0) ^{\lambda} a^*(-1)  ^{\mu} {\C}[ a (-n-1), a^{*} (-m) \ \vert \ n,m  \in {\Z}_{\ge 0} ]. \nonumber \eea

 Since $\widetilde{U( \lambda, \mu)}$ is a restricted   $\widehat{\mathcal A}$--module we get:
\begin{proposition}
$\widetilde{U( \lambda, \mu)}$ is an irreducible weight module for the Weyl vertex algebra $M$.
\end{proposition}

One can see that the weight spaces of the module $\widetilde{U( \lambda, \mu)}$ are all infinite-dimensional with respect to $(\beta(0), L(0))$. In particular,  vectors $$a(-1) ^m a^ * (0) ^{\lambda + 2m   } a^*(-1) ^{\mu- m},  \quad m \in {\Z}_{\ge 0}, $$
are linearly independent and they belong to the same weight space.

\begin{remark}
Note that modules  $\widetilde{U( \lambda, \mu)}$  are not in the category $\mathcal K$, and therefore Proposition  \ref{sc-fusion-1}  can not be applied in this case.   
\end{remark}


\section{The vertex algebra $\Pi(0)$ and the the construction of intertwining operators } \label{lattice} 

In this section we present a bosonic realization of the weight modules for the Weyl vertex algebra. We also construct  intertwining operators using this bosonic realization.

\subsection{The vertex algebra $\Pi(0)$ and its modules} Let $L$ be the lattice
$$ L= {\Z} \alpha + {\Z}\beta, \ \la \alpha , \alpha \ra = - \la \beta , \beta \ra = 1, \quad \la \alpha, \beta \ra = 0, $$
and $V_L = M_{\alpha, \beta} (1) \otimes {\C} [L]$ the associated lattice vertex superalgebra, where $M_{\alpha, \beta}(1) $ is the   Heisenberg  vertex algebra generated by fields $\alpha(z)$ and $\beta(z)$ and ${\C}[L]$ is the group algebra of $L$.
 We have its vertex subalgebra
$$ \Pi (0) = M_{\alpha, \beta} (1) \otimes {\C} [\Z (\alpha + \beta) ] \subset V_L. $$

 There is an injective vertex algebra homomorphism  $f  : M \rightarrow  \Pi(0)$ such that
$$ f(a) = e^{\alpha + \beta}, \ f(a^{*}) = -\alpha(-1) e^{-\alpha-\beta}. $$
We identify $a, a^*$ with their image in $\Pi(0)$. We have (cf. \cite{efren})
 $$ M \cong  \mbox{Ker}_{\Pi(0)} e ^{\alpha}_0.$$

The Virasoro vector $\omega$ is mapped to 
$$\omega  = a(-1) a^* (-1) {\bf 1} =  \frac{1}{2} (\alpha (-1) ^2 - \alpha(-2) - \beta(-1) ^2 + \beta(-2) ) {\bf 1}. $$
Note also that
$$ g(\omega) = - a(-2) a^* = \omega_{\mu}  =   \frac{1}{2} (\alpha (-1) ^2 - \alpha(-2) - \beta(-1) ^2 -  \beta(-2) ) {\bf 1}. \quad (\mu =1). $$



Since 
\bea && \Pi(0) = {\Bbb C}[\Z (\alpha+\beta)] \otimes M_{\alpha, \beta} (1) \label{def-pi0} \eea
is a vertex subalgebra of $V_L$, for every $\lambda \in {\C}$ and $r \in {\Z}$,
$$\Pi _r (\lambda) ={\Bbb C}[r \beta + (\Z + \lambda )(\alpha+\beta)] \otimes M_{\alpha, \beta} (1) = \Pi(0). e^{r \beta + \lambda (\alpha + \beta)}  $$
is an irreducible  $\Pi(0)$--module.

We have
\bea  L (0) e^{ r  \beta + (n + \lambda )(\alpha+\beta) } 
&=&  \frac{1-r}{2}  ( r  + 2   (n + \lambda))   e^{ r  \beta + (n + \lambda )(\alpha+\beta) },   \nonumber \eea
and for $\mu = 1$
\bea  L^{\mu}  (0) e^{ r  \beta + (n + \lambda )(\alpha+\beta) } 
&=&  \frac{1}{2} r  ( 1+  r  + 2   (n + \lambda))   e^{ r  \beta + (n + \lambda )(\alpha+\beta) }  \nonumber \eea

\begin{proposition}   \label{ired-weyl-1}Assume that $\ell  \in {\Z}$, $\lambda \in {\C} \setminus {\Z}$.  Then as $M$--modules:
\begin{itemize}
\item[(1)] $\Pi_{\ell}( \lambda) \cong    \rho_{-\ell +1} ( \widetilde{U(-\lambda)})$,
\item[(2)] $\Pi_{\ell}( \lambda)^g \cong    \rho_{\ell} ( \widetilde{U(\lambda)})$.
\end{itemize}
\end{proposition}

\begin{proof}
Assume first that $r=1$. Then $\Pi _1 (-\lambda)$ is a $\Zp$--graded $M$--module whose lowest component is
 $$ \Pi _1 (-\lambda) (0) \cong {\Bbb C}[\beta + (\Z - \lambda )(\alpha+\beta)] \cong U(\lambda).  $$ Now Lemma \ref{classification-zhu} and  Corollary  \ref{ired-relaxed} imply that $\Pi _1 (-\lambda)$ is an irreducible $M$--module isomorphic to  $\widetilde{U(\lambda)}$.
  Modules $\Pi _{\ell}  (-\lambda)$ can by obtained from $\Pi _1 (-\lambda)$ by applying the spectral flow automorphism $\rho_{-\ell+1}:=e ^{(\ell-1)\beta}$.
Using Corollary  \ref{ired-relaxed} we get
  \bea \Pi_{\ell}( \lambda)^g  &= & \left( \rho_{-\ell+1} ( \widetilde{U(-\lambda)} )\right)^g  =   \rho_1 \sigma   \rho_{-\ell +1}  ( \widetilde{U(- \lambda)} ) \nonumber \\ &=&   \rho_{\ell}   \sigma  ( \widetilde{U(-\lambda)}) 
  =    \rho_{\ell} ( \widetilde{U(\lambda)}). \nonumber \eea
  The proof follows.
  \end{proof}

\subsection{Construction of intertwining operators}
\begin{proposition} \label{konstrukcija-int-op}
For every $\ell _1, \ell _2 \in {\Z}$ and $\lambda , \mu  \in {\C}$ there exist  non-trivial  intertwining operators of types
 \bea {\rho_{\ell _1 + \ell _2-1 } ( \widetilde{U(\lambda  + \mu  )})  \choose  \rho _{\ell _1} ( \widetilde{U(\lambda )})    \ \ \rho_{\ell_2 } ( \widetilde{U(\mu )})  }, \quad 
{\rho_{\ell _1 + \ell _2  } ( \widetilde{U(\lambda  + \mu  )})  \choose  \rho _{\ell _1} ( \widetilde{U(\lambda )})    \ \ \rho_{\ell_2 } ( \widetilde{U(\mu )})  } \label{int-op} \eea
in the category of weight $M$--modules.
\end{proposition}
\begin{proof}
 By using explicit bosonic realization,  as in \cite{DL},   one can  construct a  unique nontrivial intertwining operator $I(\cdot, z)$ of type 
   \bea  &&{\Pi_{s_1 + s_2} (\lambda_1 + \lambda_2) \choose  \Pi _{s_1}  (\lambda_1)  \ \ \Pi_{s_2} (\lambda_2) } \label{int-pi} \eea  
   in the category of $\Pi(0)$--modules such that
$$ e^{ s_1 \beta + \lambda_1 (\alpha+ \beta)} _{\nu} e^{ s_2 \beta + \lambda_2 (\alpha+ \beta)} = e^{ (s_1+ s_2) \beta + (\lambda_1 +  \lambda_2)  (\alpha+ \beta)} \quad (\nu \in {\C}). $$
By restriction, this gives a non-trivial intertwining operator in the category of weight $M$--modules.
Taking the embedding $f : M \rightarrow \Pi(0)$ and applying Corollary \ref{ired-weyl-1},  we conclude the operator (\ref{int-pi}) gives the  intertwining operator of type
$$ {\rho_{-s_1 - s_2+1 } ( \widetilde{U(-\lambda_1 - \lambda_2)})  \choose  \rho _{-s_1+1} ( \widetilde{U(-\lambda_1)})    \ \ \rho_{-s_2+1} ( \widetilde{U(-\lambda_2)})  }, $$
which for $\ell_1 = -s_1 +1$, $\ell_2 = -s_2+1$, $\lambda= -\lambda_1$, $\mu = -\lambda_2$ gives the  first intertwining operator.
By using   action of the automorphism $g$ and Corollary  \ref{ired-weyl-1} we get the following intertwining operator
$$ {\rho_{s_1 + s_2 } ( \widetilde{U(\lambda_1 + \lambda_2)})  \choose  \rho _{s_1} ( \widetilde{U(\lambda_1)})    \ \ \rho_{s_2} ( \widetilde{U(\lambda_2)})  }, $$
which for $\ell_1 = s_1$, $\ell_2 = s_2$, $\lambda= \lambda_1$, $\mu = \lambda_2$ gives the  second intertwining operator.
\end{proof} 
 \begin{remark} Intertwining operators in  Proposition \ref{konstrukcija-int-op} are realized on irreducible $\Pi(0)$--modules. This result can be read as
 $$ \Pi _{\ell _1}  (\lambda) \times \Pi_{\ell_2} (\mu)  \supseteq  \Pi_{\ell _1 + \ell_2-1} (\lambda  + \mu) + \Pi_{\ell _1 + \ell_2} (\lambda  + \mu). $$
 
In the category of $M$--modules, we have non-trivial intertwining operators 
  \bea  &&{\Pi_{\ell _1 + \ell_2-1} (\lambda  + \mu ) \choose  \Pi _{\ell _1}  (\lambda)  \ \ \Pi_{\ell_2} (\mu) },  \label{int-pi-2} \eea 
  which are not  $\Pi(0)$--intertwining operators. 
 \end{remark}
 
 \begin{remark}
 Note that $(M, Y, \textbf{1})$ is a conformal vertex algebra with the conformal vector 
 \[\omega = a(-1)a^{*}(-1)\textbf{1}.\]
 Note that the intertwining operators (\ref{int-op}) satisfy the $L(-1)$-derivative property. Intertwining operators (\ref{int-pi}) satisfy this property by using the lattice realization as before, and intertwining operators (\ref{int-pi-2}) satisfy it by    Proposition \ref{auto}, using the facts that $g(\omega) = \omega_1$ and $L^{\mu}(-1) = L(-1)$, for $\mu = 1$. Moreover,  using  relation (\ref{der-mu}) we see that the $L^{\mu} (-1)$--derivation property holds for every $\mu \in \mathbb{C} $, for all intertwining operators constructed above.
 \end{remark}

\section{ The vertex algebra $V_1 (gl(1 \vert 1))$ and its modules}

   \subsection{On the vertex algebra $V_1(gl(1 \vert 1))$.}
We now recall some results on the representation theory of $gl(1 \vert 1)$ and $\widehat{ gl(1\vert 1)}$. The terminology follows \cite[Section 5]{CR1}.

Let  $\g = gl(1 \vert 1)$  be the complex Lie superalgebra generated by two even elements $E$, $N$  and two odd elements $\Psi^{\pm}$ with the following  (super)commutation relations:
$$ [\Psi^+, \Psi^{-}] = E, \ [E, \Psi^{\pm}] = [E, N] = 0, \ [N, \Psi^{\pm}] = \pm \Psi^{\pm}. $$
Other (super)commutators are trivial.
Let $ (\cdot, \cdot)$ be the invariant  super-symmetric invariant bilinear form such that
$$ ( \Psi^+, \Psi^-) = ( \Psi^-, \Psi^+) =  1,  \ (N, E) = (E, N) = 1.$$
All other products are zero.

Let $\hg = \widehat{gl(1 \vert 1)}= {\g} \otimes {\C}[t,t^{-1}] + {\C} K$  be the associated  affine Lie superalgebra  with  the   commutation relations
$$ [x (n), y(m)] = [x, y](n+m) + n \delta_{n+m,0} K, $$ 
where $K$ is central and  for $x \in {\g}$  we set $x(n) = x \otimes t^n$.
Let $V_k(\g)$ be the associated  simple  affine vertex algebra of level $k$.

Let $\mathcal{V} _{r,s}$ be the Verma module for the Lie superalgebra $\g$ generated by the vector $v_{r,s}$ such that $ N v_{r,s} = r v_{r, s}$, $ E v_{r,s} = s v_{r,s}$. This module is a $2$--dimensional module  and it is irreducible iff $s \ne 0$. If $s= 0$, $\mathcal{V} _{r,s}$ has a $1$--dimensional irreducible quotient, which we denote by  $\mathcal{A}_{r}$.

We will need the following tensor product decompositions:
\bea  && \mathcal{A}_{r_1}  \otimes \mathcal{A}_{r_2} = \mathcal{A}_{r_1+ r_2} , \quad  \mathcal{A}_{r_1} \otimes \mathcal{V} _{r_2,s_2} = \mathcal{V} _{r_1 + r_2 , s_2}, \label{dek-1} \\
&&  \mathcal{V} _{r_1,s_1} \otimes \mathcal{V} _{r_2,s_2} = \mathcal{V} _{r_1 + r_2  ,s_1+s_2} \oplus \mathcal{V} _{r_1+r_2 - 1 ,s_1 + s_2} \quad (s_1 + s_2 \ne 0), \label{dek-2} \\
&&  \mathcal{V} _{r_1,s_1} \otimes \mathcal{V} _{r_2,-s_1} = \mathcal{P} _{r_1 + r_2 }, \label{dek-3} \eea
where 
$\mathcal{P}_r$ is the $4$--dimensional  indecomposable module which appears in the following extension
$$ 0 \rightarrow \mathcal{V} _{r  ,0} \rightarrow \mathcal{P}_r \rightarrow  \mathcal{V} _{r -1,0} \rightarrow 0. $$
Let $\widehat {\mathcal{V} } _{r,s}$ denote the Verma module of level $1$ induced from the irreducible $gl(1 \vert 1)$--module $\mathcal{V} _{r ,s }$. If $s\notin \Z$, then  $\widehat {\mathcal{V} } _{r,s}$ is an irreducible $V_1 (gl(1 \vert 1))$--module. If $s\in {\Z}$, $\widehat {\mathcal{V} } _{r,s}$ is reducible and its structure is described in \cite[Section 5.3]{CR1}.

By using the tensor product decomposition (\ref{dek-2})  we get  the following  result on fusion rules   for  $V_1( \g)$--modules:

\begin{proposition} \label{non-gl11}
Let  $r_1, r_2, s_1, s_2 \in {\C}$, $s_1, s_2, s_1 + s_2 \notin {\Z}$. Assume that there is a non-trivial intertwining operator 
$$ {  \widehat {\mathcal{V} } _{r_3,s_3} 
\choose \widehat {\mathcal{V} } _{r_1,s_1} \ \   \widehat {\mathcal{V} } _{r_2,s_2}  } $$
in the category of $V_1(\g)$--modules. Then $s_3 = s_1 + s_2$ and $r_3 = r_1 + r_2$, or $r_3 = r_1 + r_2 -1$.
\end{proposition}

 Recall that the  Clifford algebra is generated by $\psi(r), \psi^{*}(s)$, where $r, s \in \mathbb{Z} + \frac{1}{2}$, with relations 
\begin{gather}
\begin{aligned}
& [\psi(r), \psi^{*}(s)] = \delta_{r+s, 0}, \label{rel}\\
& [\psi(r), \psi(s)] = [\psi^{*}(r), \psi^{*}(s)] = 0, \text{ for all } r, s.
\end{aligned}
\end{gather}
Note that the commutators (\ref{rel}) are actually anticommutators because both $\psi(r)$ and $\psi^{*}(s)$ are odd for every $r$ and $s$.

The Clifford vertex algebra $F$  is generated by the fields 
\begin{align*}
& \psi(z) = \sum_{n \in \mathbb{Z}} \psi(n+\frac{1}{2})z^{-n-1}, \\
& \psi^{*}(z) = \sum_{n \in \mathbb{Z}} \psi^{*}(n+\frac{1}{2})z^{-n-1}.
\end{align*}
As a vector space, \[F \cong \bigwedge\big(\big\{\psi(r), \psi^{*}(s)\ \big\vert\ r,s < 0\big\}\big)\]



 Let $ V_{\Z \gamma}$ be the lattice vertex algebra associated to the lattice $\Z \gamma\cong \Z $, $\langle \gamma, \gamma \rangle = 1$. By using the  boson--fermion correspondence,  we have that  $F \cong V_{\Z \gamma}$,  and we can identify  the generators of the Clifford vertex algebra as follows (cf. \cite{K2}):
\[ \psi:= e^{\gamma}, \psi^{*} =  e^{-\gamma}\].

Now we define the following vertex superalgebra:
$$ S\Pi (0) = \Pi (0) \otimes F \subset V_L,$$
and its irreducible modules
$$ S\Pi _r(\lambda) = \Pi_r (\lambda) \otimes F = S\Pi(0). e^{r \beta + \lambda (\alpha + \beta)}. $$

  Let $\mathcal U = M \otimes F$.  Using \cite[Section 5.8]{K2} we define the vectors  \bea
&&  \Psi ^+ := e ^{\alpha + \beta + \gamma} = a(-1) \psi, \ \Psi^- :=- \alpha(-1) e ^{-\alpha - \beta - \gamma} =  a^* (0) \psi ^*, \nonumber \\ &&  E := \gamma + \beta, \
N:= \frac{1}{2} (\gamma - \beta). \nonumber \eea
Then the components of the fields 
$$ X  (z) = Y(X , z) = \sum_{n \in {\Z}} X (n) z^{-n-1}, \ \ X\in \{ \Psi^+, \Psi^-, E, N\}$$
satisfy the commutation relations for the affine Lie algebra $\hg = \widehat{gl(1 \vert 1)}$, so that $M \otimes F$ is a $\hg$--module of level $1$.
(See also \cite{A-2017} for a realization of $\widehat{gl(1 \vert 1)}$ at the critical level).

The Sugawara conformal vector is 
\bea  \ \omega_{c=0} &= & \tfrac{1}{2} ( N (-1) E (-1) + E (-1) N (-1) - \Psi^+ (-1) \Psi^- (-1)  + \nonumber   \\ && \quad  \Psi^- (-1) \Psi^+  (-1)  + E(-1) ^2 ) {\bf 1}\label{expression-omega-1}\   \\
&=&   \tfrac{1}{2}  (\beta(-1) + \gamma(-1)) (\gamma(-1) - \beta(-1))  +  \alpha(-1) ( \alpha (-1) + \beta(-1) + \gamma(-1))  \nonumber \\
& &   - \tfrac{1}{2} ( ( \alpha (-1) + \beta(-1) + \gamma(-1) )^2 + (\alpha(-2) + \beta(-2) + \gamma(-2)) )  \nonumber \\
& & + \tfrac{1}{2}  (\beta (-1)  + \gamma(-1) ) ^2  + \tfrac{1}{2} (\beta (-2) + \gamma(-2) ) \nonumber \\
&= &\tfrac{1}{2} (\alpha(-1) ^2 -  \alpha(-2)   -\beta (-1) ^2 + \gamma(-1) ^2  ) \nonumber   \\
&=&\omega_{c=-1} + \tfrac{1}{2} \gamma(-1) ^2 \quad (\omega_{c=-1} = \omega_{1/2}) \nonumber  \eea

\subsection{Construction of irreducible  $V_1 (\g)$--modules from irreducible $M$--modules.}

Let $V_1 (\g)$  be  the simple affine vertex algebra of level $1$ associated to $\g$.

 We have the following gradation:
 $$ \mathcal U = \bigoplus \mathcal U ^{\ell}, \quad E(0) \vert _{ \mathcal U ^{\ell} } = \ell  \ \mbox{Id}. $$
 We will present an alternative  proof of the following  result:

\begin{proposition}  \label{simpl-gl11} \cite{K2}
We have:
$$ V_1 (\g) \cong \mathcal U^0 =\mbox{Ker}_{M \otimes F}  E(0) . $$
\end{proposition}
 \begin{proof}
   Let $ \widetilde V_1 (\g)$ be the vertex subalgebra of  $\mathcal U^0$ generated by $\g$. Assume that $ \widetilde V_1 (\g) \ne \mathcal U^0$. Then there is a subsingular vector   $v_{r,s} \notin {\C}{\1}$   for  $\hg$ of weight $(r, s)$ such that for $n > 0$:
  \bea &&   \Psi ^+ (0)  v_{r,s} \in \widetilde V_1 (\g),  \quad  X(n) v_{r,s} \in  \widetilde V_1 (\g), \quad X \in \{ E, N, 
\Psi^{\pm}\}\nonumber \\
&& E(0) v_{r,s} = s v_{r,s}, \ N(0) v_{r,s} = r v_{r,s}.\nonumber \eea
 In other words, $v_{r,s}$ is a singular vector  in  the quotient $\widetilde { \mathcal U ^0 }= \mathcal U ^0 /  \widetilde V_1 (\g)$.  Since $E(0)$ acts trivially on $\mathcal U ^{0}$, we conclude that $s=0$. Recalling the expression for the Virasoro conformal vector (\ref{expression-omega-1}), we get that in $\widetilde { \mathcal U ^0 }$:
 $$ L^{c=0} (0) v_{r,0} = ( \omega_{c=0} )_1 v_{r,0} = \tfrac{1}{2} ( 2  N (0) E (0)   - E(0) + E(0) ^2 ) v_{r,0} = 0. $$
 This implies that $v_{r,0}$ has the conformal weight $0$ and hence must be proportional to ${\bf 1}$. A contradiction.  Therefore, $\mathcal U ^0  =  \widetilde V_1 (\g)$. Since $\mathcal U^0$ is a simple vertex algebra,
 we have that $ \widetilde V_1 (\g) =   V_1 (\g)$.
 \end{proof}

 We can extend this irreducibility result to a wide class of weight modules. 
The proof is  similar to the one given in \cite[Theorem 6.2]{A-2007}.
 
 \begin{theorem}  \label{ired-general} Assume that $\mathcal N$ is an irreducible weight module for the Weyl vertex algebra $M$, such that $\beta(0)$ acts semisimply on $\mathcal N$:
 $$ \mathcal N = \bigoplus_{s \in {\Z} + \Delta} \mathcal N^s, \quad \beta(0) \vert N^s \equiv s \mbox{Id} \quad (\Delta \in {\C}). $$
 Then $\mathcal N \otimes F$ is a completely reducible $V_1(\g)$--module:
 $$ \mathcal N \otimes F = \bigoplus_{s \in {\Z}}  \mathcal L_s(N)  \quad  \mathcal L_s (N) = \{ v \in    \mathcal N \otimes F  \ \vert  E(0) v = ( s + \Delta) v \},$$ 
 and each $ \mathcal L_s(N) $ is irreducible $V_1(\g)$--module.
 \end{theorem}
 \begin{proof}
 Clearly  $\mathcal L_s (N)$ is a $\mathcal U^0$$ (= V_1(\g))$--module.  It suffices to prove that each vector $ w \in \mathcal L_s (N) $ is cyclic.
Since $ \mathcal N \otimes F$ is a simple $\mathcal U$--module, we have that $\mathcal U. w = \mathcal N \otimes F$. On the other hand, $ \mathcal N \otimes F $ is ${\Z}$--graded $\mathcal U$--module so that
$$ \mathcal U ^{r} \cdot \mathcal L_{s} (N) \subset \mathcal L_{r+ s}(N), \quad (r, s \in {\Z}). $$
This implies that $\mathcal U^r . w \subset \mathcal L_{r+s} (N)$ for each $r \in {\Z}$. Theferore $\mathcal U^0 . w  = \mathcal L_{r} (N)$. The proof follows.
 \end{proof}

As a consequence we get a family of irreducible $V_1 (\g)$--modules:
 \begin{corollary} Assume that $\lambda, \mu \in {\C} \setminus {\Z}$. Then for each $s \in {\Z}$ we have:
 \begin{itemize}
 \item[(1)] $\mathcal L_s (\widetilde{ U( \lambda)})$ is an irreducible $V_1(\g)$--module,
 \item[(2)]  $\mathcal L_s (\widetilde{ U( \lambda, \mu )})$ is an irreducible $V_1(\g)$--module.
 \end{itemize} 
 \end{corollary}
We will prove  in the next section that  $\mathcal L_s (\widetilde{ U( \lambda)})$ are irreducible highest weight modules. But one can see  that  $\mathcal L_s (\widetilde{ U( \lambda, \mu )})$ have infinite-dimensional weight spaces. A detailed analysis of  the structure of these modules will appear in our forthcoming papers (cf. \cite{AdP-2019}).

 \section{The calculation of fusion rules}
 
 In this section we will finish the calculation of fusion rules for the Weyl vertex algebra $M$. 
 
 We will first identify certain irreducible highest weight $\hg$--modules. 

\begin{lemma}  \label{identif-lema} Assume that $r, n \in {\Z}$, $\lambda \in {\C} $, $\lambda+ n \notin {\Z}$. Then
$e ^{ r (\beta + \gamma)  + (\lambda +n) (\alpha + \beta) }$ is a singular vector in $S\Pi _r(\lambda)$ and
\bea &&  U (\hg).  e ^{ r (\beta + \gamma)  + (\lambda +n) (\alpha + \beta) } \cong \widehat{ \mathcal{V} }_{  r + \tfrac{1}{2} ( \lambda + n ), -\lambda -n} ,  \label{identif-h-w}\\
&&  L(0)  e ^{ r (\beta + \gamma)  + (\lambda +n) (\alpha + \beta) }  = \frac{1}{2}  (1-2 r) (n+\lambda)   e ^{ r (\beta + \gamma)  + (\lambda +n) (\alpha + \beta) }.   \label{conf-w-m} \eea
 \end{lemma}
 \begin{proof}
 By using standard calculation in lattice vertex algebras we get for $m \ge 0$
 \bea
 \Psi^+ (m ) e ^{ r (\beta + \gamma)  + (\lambda +n) (\alpha + \beta) }  &=& e^{\alpha+ \beta + \gamma}_{m} e ^{ r (\beta + \gamma)  + (\lambda +n) (\alpha + \beta) } = 0, \nonumber \\
  \Psi^- (m+1) e ^{ r (\beta + \gamma)  + (\lambda +n) (\alpha + \beta) }  &=& \left(  -\alpha(-1)  e^{-\alpha- \beta - \gamma} \right)_{m+1} e ^{ r (\beta + \gamma)  + (\lambda +n) (\alpha + \beta) } = 0, \nonumber \\
  E(m) e ^{ r (\beta + \gamma)  + (\lambda +n) (\alpha + \beta) }  &=& - (\lambda + n) \delta_{m ,0} e ^{ r (\beta + \gamma)  + (\lambda +n) (\alpha + \beta) },   \nonumber \\
  N(m)  e ^{ r (\beta + \gamma)  + (\lambda +n) (\alpha + \beta) }  &=&     \frac{1}{2} ( 2 r +  \lambda + n )  \delta_{m ,0} e ^{ r (\beta + \gamma)  + (\lambda +n) (\alpha + \beta) }. \nonumber  
 \eea
 Therefore  $e ^{ r (\beta + \gamma)  + (\lambda +n) (\alpha + \beta) } $ is a highest weight vector for $\hg$ with highest weight    $(r + \tfrac{1}{2} ( \lambda + n ), -\lambda -n)$ with respect to $(N(0), E(0))$. This implies that $U (\hg )$ is isomorphic to a certain quotient of the Verma module $\widehat{ \mathcal{V} }_{  r + \tfrac{1}{2} ( \lambda + n ), -\lambda -n} $. But since,  $\lambda+ n \notin {\Z}$, $\widehat{ \mathcal{V} }_{  r + \tfrac{1}{2} ( \lambda + n ), -\lambda -n} $ is irreducible and therefore   (\ref{identif-h-w}) holds. Relation (\ref{conf-w-m}) follows by  applying the expression $\omega = \frac{1}{2} (\alpha(-1) ^2 - \alpha(-2) - \beta(-1) ^2 + \gamma(-1) ^2)$: 
 \bea  && L(0) e ^{ r (\beta + \gamma)  + (\lambda +n) (\alpha + \beta) }   \nonumber \\ = && \frac{1}{2} \left( (\lambda+n) ^2 - (\lambda + n + r)^2 + r^2 +  (\lambda + n)   \right)  e ^{ r (\beta + \gamma)  + (\lambda +n) (\alpha + \beta) } \nonumber   \\
=  &&  \frac{1}{2}  (1-2 r) (n+\lambda)   e ^{ r (\beta + \gamma)  + (\lambda +n) (\alpha + \beta) }. \nonumber  \eea
 \end{proof}
 
\begin{theorem}
Assume that $r \in {\Z}$, $\lambda \in {\C} \setminus {\Z}$. Then we have:
\item[(1)] $S\Pi _r(\lambda)$ is an irreducible $M \otimes F$--module,

\item[(2)] $S\Pi _r(\lambda)$ is a completely reducible $\widehat{ gl(1\vert 1)}$--module:
\bea  S\Pi _r(\lambda ) &\cong& \bigoplus _{ s \in {\Z} } U (\hg). e ^{ r (\beta + \gamma)  + (\lambda +s) (\alpha + \beta) } \nonumber \\
&\cong& \bigoplus _{ s \in {\Z} } \widehat {\mathcal{V}} _{   r + \tfrac{1}{2} ( \lambda + s ), -\lambda -s} . \label{dec-int-12}\eea

\end{theorem}

\begin{proof}
 The assertion (1) follows from the fact that $\Pi _r(\lambda)$  is an irreducible $M$--module (cf. Proposition \ref{ired-weyl-1}). Note next that the operator $E(0) = \beta (0) +\gamma(0)$ acts semi--simply on $M \otimes F$:
 $$ M \otimes F = \bigoplus  _{s\in {\Z}}  ( M \otimes F ) ^ {(s)}, \quad  ( M \otimes F ) ^ {(s)} = \{ v \in M \otimes F  \vert \ E(0) v = -s v \}. $$
 In particular, $  ( M \otimes F ) ^ {(0)} \cong V_1 (\g)$ (cf. \cite{K2} and Proposition \ref{simpl-gl11}).
 But $E(0) $ also defines the following  $\Z$--gradation on $S\Pi _r(\lambda)$:
 $$S\Pi _r(\lambda)  = \bigoplus  _{s\in {\Z}} S\Pi _r(\lambda) ^ {(s)}, \quad S\Pi _r(\lambda) ^ {(s)} = \{ v \in S\Pi _r(\lambda)   \vert \ E(0) v = (-s - \lambda) v \}. $$
 Applying Theorem \ref{ired-general}  we see that each  $S\Pi _r(\lambda) ^ {(s)}$ is an irreducible  $(M \otimes F ) ^ {(0)} \cong V_1 (\g)$--module.  Using Lemma  \ref{identif-lema} we see that it is an irreducible  highest weight $\hg$--module with highest weight
 vector $e ^{ r (\beta + \gamma)  + (\lambda +s) (\alpha + \beta) }$. The proof follows.
\end{proof}

 \begin{theorem} \label{fusion-rules-1}
 Assume that $\lambda_1, \lambda_2, \lambda_1 + \lambda_2 \in {\C}\setminus {\Z}$, $r_1, r_2, r_3 \in {\Z}$.
 Assume that there is a non-trivial intertwining operator of type
$$ { S \Pi_{r_3} (\lambda_3) \choose S\Pi_{r_1} (\lambda_1) \ \  S\Pi_{r_2} (\lambda_2) } $$
in the category of $M\otimes F$--modules.
 Then $\lambda_3 = \lambda_1 + \lambda_2$ and $r_3 = r_1 + r_2$, or $r_3 = r_1 + r_2 - 1 $.
 \end{theorem}
 \begin{proof}
 Assume that $I$ is an non-trivial intertwining operator of type $$ { S \Pi_{r_3} (\lambda_3) \choose S\Pi_{r_1} (\lambda_1) \ \  S\Pi_{r_2} (\lambda_2) }, $$ Since $S\Pi_r (\lambda)$ are simple $M \otimes F$--modules, we have that for every $s_1, s_2 \in {\Bbb Z}$:
 $$ I  ( e ^{ r_1 (\beta + \gamma)  + (\lambda_1 +s_1 ) (\alpha + \beta) }  , z)   e ^{ r_2 (\beta + \gamma)  + (\lambda_2 +s_2 ) (\alpha + \beta) }  \ne 0. $$
Here we use the well-known result which states that for every non-trivial  intertwining operator $I$  between three irreducible modules we have that  $I(v,z) w \ne 0$ (cf. \cite[Proposition 11.9]{DL}).
 Note that $e ^{ r_i (\beta + \gamma)  + \lambda_i(\alpha + \beta) }$   is a singular vector for $\hg$ which generates $V_1(\g)$--module  $\widehat{\mathcal{V}}_{  r_i + \tfrac{1}{2}  \lambda_i, -\lambda_i}$, $i=1,2$. The restriction of $I(\cdot,z)$ on 
$$ \widehat{\mathcal{V}}_{  r_1 + \tfrac{1}{2}  \lambda_1, -\lambda_1} \otimes  \widehat{\mathcal{V}}_{  r_2 + \tfrac{1}{2}  \lambda_2, -\lambda_2} $$
gives a non-trivial intertwining operator 
$$ {  S \Pi_{r_3} (\lambda_3) 
\choose \widehat{\mathcal{V}}_{  r_1 + \tfrac{1}{2}  \lambda_1, -\lambda_1}  \ \    \widehat{\mathcal{V}}_{  r_2 + \tfrac{1}{2}  \lambda_2, -\lambda_2}  } $$
in the category of $V_1(\g)$--modules.
Proposition \ref{non-gl11} implies that then  $$\widehat{\mathcal{V}}_{  r_1 + r_2 + \tfrac{1}{2}  ( \lambda_1 + \lambda_2) , -\lambda_1- \lambda_2} \quad \mbox{or} \quad \widehat{\mathcal{V}}_{  r_1 + r_2 + \tfrac{1}{2}  ( \lambda_1 + \lambda_2) -1 , -\lambda_1- \lambda_2} $$
has to appear in the decomposition of  $S \Pi_{r_3} (\lambda_3) $ as a $V_1(\g)$--module.
Using decomposition (\ref{dec-int-12}) we get that there is $s \in {\Z}$ such that 
\bea &&  r_1 + r_2 + \tfrac{1}{2}  ( \lambda_1 + \lambda_2)  =  r_3 + \tfrac{1}{2} ( \lambda + s ), \quad   -\lambda_1- \lambda_2  = -\lambda -s \label{jedn-prva} \eea
or 
\bea && r_1 + r_2- 1 + \tfrac{1}{2}  ( \lambda_1 + \lambda_2)  =  r_3 + \tfrac{1}{2} ( \lambda + s ), \quad   -\lambda_1- \lambda_2  = -\lambda -s . \label{jedn-druga} \eea
Solution of  (\ref{jedn-prva}) is  $$ \lambda + s= \lambda_1 + \lambda_2, \ r_3  =r_1 + r_2, $$ and of (\ref{jedn-druga})   is $$ \lambda + s= \lambda_1 + \lambda_2,  \quad r_3  =r_1 + r_2-1.$$
Since $S\Pi _r(\lambda)  \cong S\Pi _r(\lambda + s) $ for $ s \in {\Z}$, we can take $s = 0$. Thus,
$\lambda_3 = \lambda_1 + \lambda_2$ and $r_3 = r_1 + r_2$ or $r_3 = r_1 + r_2 - 1$.
The claim holds.
%
 %
 \end{proof}

By using the following natural isomorphism of the spaces of intertwining operators (cf. \cite[Section 2]{ADL}): $$ \mbox{I}_{M \otimes F} { S \Pi_{r_3} (\lambda_3) \choose S\Pi_{r_1} (\lambda_1) \ \ S\Pi_{r_2} (\lambda_2) } \cong   \mbox{I}_{M } {  \Pi_{r_3} (\lambda_3) \choose \Pi_{r_1} (\lambda_1) \ \  \Pi_{r_2} (\lambda_2) },$$
Theorem \ref{fusion-rules-1} implies the fusion rules result in the category of modules for the Weyl vertex algebra $M$ (see also \cite[Corollary 6.7]{RW}, for a derivation of the same fusion rules using Verlinde formula).

\begin{corollary} Assume that $\lambda_1, \lambda_2, \lambda_1 + \lambda_2 \in {\C}\setminus {\Z}$, $r_1, r_2, r_3 \in {\Z}$.
 There exists  a non-trivial intertwining operator of type
$$ {  \Pi_{r_3} (\lambda_3) \choose \Pi_{r_1} (\lambda_1) \ \  \Pi_{r_2} (\lambda_2) } $$
in the category of $M$--modules 
 if and only if  $\lambda_3 = \lambda_1 + \lambda_2$ and $r_3 = r_1 + r_2$ or $r_3 = r_1 + r_2 - 1 $.
 
 The fusion rules in the category of weight $M$--modules are given by
 $$ \Pi_{r_1} (\lambda_1)  \times \Pi_{r_2} (\lambda_2)  =  \Pi_{r_1+r_2 } (\lambda_1+ \lambda_2)  + \Pi_{r_1+ r_2-1} (\lambda_1 + \lambda_2).  $$
\end{corollary}


 


\footnotesize{
  \noindent{\bf D.A.}:  Department of Mathematics, Faculty of Science, University of Zagreb, Bijeni\v{c}ka 30, 10 000 Zagreb, Croatia;\newline
{\tt adamovic@math.hr}

 \noindent{\bf V.P.}:  Department of Mathematics, Faculty of Science, University of Zagreb, Bijeni\v{c}ka 30, 10 000 Zagreb, Croatia;\newline
{\tt vpedic@math.hr}
 
}


\begin{thebibliography}{FGST2}

\bibitem{A-2001} D Adamovi\' c, Vertex algebra approach to fusion rules for N = 2 superconformal minimal models, J. Algebra, 239 (2001), 549–572.
 
\bibitem{A-2003}  D. Adamovi\' c, Classification of irreducible modules of certain subalgebras of free boson vertex algebra, J. Algebra 270 (2003), no. 1, 115-132.

\bibitem{A-2007} D. Adamovi\' c,   Lie superalgebras and irreducibility of $A_1^{(1)}$- modules at the critical level , Commun. Math. Phys. 270 (2007), 141--161.

\bibitem{A-CMP} D. Adamovi\' c, Realizations of simple affine vertex algebras and their modules: the cases $\widehat{sl(2)}$ and $\widehat{osp(1,2)}$,  Commun. Math. Phys. 366  (2019)   1025--1067,  arXiv:1711.11342 [math.QA].

\bibitem{A-2017} D Adamovi\' c,  A note on the affine vertex algebra associated to $\mathfrak{gl}(1\vert1)$ at the critical level and its generalizations, Rad HAZU, Matematičke znanosti, Vol. 21 (2017), 75-87

\bibitem{ADL} T. Abe, C. Dong and H. Li, Fusion rules for the vertex operator algebras $M(1)^+$ and $V_L^+$,  Commun. Math. Phys. 253 (2005), no. 1, 171--219.

\bibitem{AdM-2017} D. Adamovi\' c, A. Milas,  Some applications and constructions of intertwining operators in logarithmic conformal field theory, Contemp. Math., 695, 15--27.

\bibitem{ALPY-2018} D. Adamovi\'  c,  C. H. Lam, V. Pedi\' c, N. Yu, On irreducibility of modules of Whittaker type for cyclic orbifold vertex algebra, arXiv:1811.04649 [math.QA]

\bibitem{AdP-2019} D. Adamovi\'  c,   V. Pedi\' c, On weight and Whittaker modules for $\widehat{gl(1\vert 1)}$, in preparation.
 
  

\bibitem{CHY} T. Creutzig, Y.-Z. Huang, J. Yang,  Braided Tensor Categories of Admissible Modules for Affine Lie Algebras, Commun. Math.  Phys. 362 (2018), 827-854  

\bibitem{CRTW} T. Creutzig, T. Liu, D. Ridout, S. Wood, Unitary and non-unitary $N=2$ minimal models, arXiv:1902.08370

\bibitem{CM} T. Creutzig and A. Milas,  False theta functions and the Verlinde formula, Adv. Math. 262 (2014), 520--545. 

\bibitem{CR1} T. Creutzig and D. Ridout, Logarithmic Conformal Field Theory: Beyond an Introduction, J. Phys. A 46 (2013), no. 49, 494006
  

 \bibitem{DLM}
C.~Dong,  H.~Li, G.~Mason,
Simple currents and extensions of vertex operator algebras, Commun. Math. Phys., 180 (1996), 671--707.

\bibitem{DL}  C. Dong, J. Lepowsky: Generalized vertex algebras and relative vertex operators, Boston:
Birkhäuser, (1993)



\bibitem {efren} E. Frenkel, Lectures on Wakimoto modules,
opers and the center at the critical level, Adv. Math 195 (2005)
297-404.

\bibitem {FB} E. Frenkel and D. Ben-Zvi, Vertex algebras and
algebraic curves, Math. Surveys Monogr., no. 88,
AMS, 2000.


\bibitem{FGM} V. Futorny, D. Grantcharov, and V. Mazorchuk, \emph{Weight
modules over infinite dimensional Weyl algebras}, Proc. Amer. Math.
Soc. \textbf{142} (2014), no. 9, 3049\textendash 3057.
 

\bibitem{FHL} I. B. Frenkel, Y.-Z. Huang and J. Lepowsky, On axiomatic approaches to vertex
operator algebras and modules, preprint, 1989, Mem. Amer. Math. Soc. 104,
(1993).
   
 \bibitem{FZ} I.~Frenkel, Y.~Zhu, 
Vertex operator algebras associated to representations of affine and Virasoro algebras, Duke Math. J. 66 (1992), 123--168

\bibitem{H} Y.-Z. Huang, Vertex operator algebras and the Verlinde conjecture, Commun. Contemp. Math. 10 (2008), no. 1, 103--154. 


\bibitem {K2}  V. Kac,   Vertex Algebras for Beginners, University
Lecture Series, Second Edition,  Amer. Math. Soc., 1998, Vol. 10.

\bibitem{KR} V. Kac, A. Radul, Representation theory of the vertex algebra $W_{1+ \infty}$,  Transform.
Groups 1 (1996) 41--70, hep-th/9512150

\bibitem{KR-CMP}K. Kawasetsu,  D. Ridout, Relaxed highest-weight modules I: rank 1 cases, Commun. Math. Phys. (2019), arXiv:1803.01989 [math.RT]

\bibitem{Li-1997} H. Li, The phyisical superselection principle in vertex operator algebra theory,  J. Algebra 196 (1997)  436--457.

\bibitem{LL} J. Lepowsky and H. Li, Introduction to vertex operator algebras and their representations, Progress in Mathematics, vol. 227, Birkh\" auser Boston, Inc., Boston, MA, 2004

 \bibitem{RW} D. Ridout, S. Wood,  Bosonic Ghosts at $c=2$ as a Logarithmic CFT,  Lett. Math. Phys., Volume 105 (2015), Issue 2, 279--307, arXiv:1408.4185

\end{thebibliography}
\end{document}